\documentclass[12pt,fleqn]{article}

\usepackage{amsmath,amsthm,amssymb}
\usepackage{geometry} 
\usepackage[pdfpagemode=UseNone,pdfstartview=FitH]{hyperref}

\newcommand{\A}{{\cal A}}
\newcommand{\Ap}[1][]{A_p\, #1}
\newcommand{\Bp}[1][]{B_p\, #1}
\newcommand{\bdry}[1]{\partial #1}
\newcommand{\D}{{\cal D}}
\newcommand{\dualp}[3][]{\left(#2,#3\right)_{#1}}
\newcommand{\F}{{\cal F}}
\newcommand{\incl}{\subset}
\newcommand{\N}{\mathbb N}
\newcommand{\norm}[2][]{\left\|#2\right\|_{#1}}
\renewcommand{\o}{\text{o}}
\newcommand{\PS}[1]{$(\text{PS})_{#1}$}
\newcommand{\R}{\mathbb R}
\newcommand{\RP}{\R \text{P}}
\newcommand{\seq}[1]{\left(#1\right)}
\newcommand{\set}[1]{\left\{#1\right\}}
\newcommand{\vol}[1]{\left|#1\right|}
\newcommand{\Z}{\mathbb Z}

\newenvironment{enumroman}{\begin{enumerate}

}{\end{enumerate}}

\newtheorem{corollary}{Corollary}[section]
\newtheorem{theorem}[corollary]{Theorem}

\theoremstyle{remark}
\newtheorem{remark}[corollary]{Remark}

\numberwithin{equation}{section}

\title{\bf Abstract multiplicity theorems and applications to critical growth problems\thanks{{\em MSC2010:} Primary 58E05, Secondary 35J92, 35B33
\newline \indent\; {\em Key Words and Phrases:} Abstract multiplicity theorems, critical growth $p$-Laplacian and $(p,q)$-Laplacian problems, arbitrarily many solutions, eigenvalues based on the cohomological index, piercing property}}
\author{\bf Kanishka Perera\\
Department of Mathematical Sciences\\
Florida Institute of Technology\\
150 W University Blvd, Melbourne, FL 32901-6975, USA\\
\em kperera@fit.edu}
\date{}

\begin{document}

\maketitle

\begin{abstract}
We prove some abstract multiplicity theorems that can be used to obtain multiple nontrivial solutions of critical growth $p$-Laplacian and $(p,q)$-Laplacian type problems. We show that the problems considered here have arbitrarily many solutions for all sufficiently large values of a certain parameter $\lambda > 0$. In particular, the number of solutions goes to infinity as $\lambda \to \infty$. Moreover, we give an explicit lower bound on $\lambda$ in order to have a given number of solutions. This lower bound is in terms of a sequence of eigenvalues constructed using the $\Z_2$-cohomological index. This is a consequence of the fact that our abstract multiplicity results make essential use of the piercing property of the cohomological index, which is not shared by the genus.
\end{abstract}

\section{Introduction}

In the celebrated paper \cite{MR709644}, Br{\'e}zis and Nirenberg showed that the problem
\begin{equation} \label{1.5}
\left\{\begin{aligned}
- \Delta u & = \lambda\, |u|^{r - 2}\, u + |u|^{2^\ast - 2}\, u && \text{in } \Omega\\[10pt]
u & = 0 && \text{on } \bdry{\Omega},
\end{aligned}\right.
\end{equation}
where $\Omega$ is a bounded domain in $\R^N$, $N \ge 3$, $2 < r < 2^\ast$, $2^\ast = 2N/(N - 2)$ is the critical Sobolev exponent, and $\lambda > 0$ is a parameter, has a positive solution for all sufficiently large $\lambda > 0$ when $N = 3$ and for all $\lambda > 0$ when $N \ge 4$. There is now a large literature extending this result to various critical growth elliptic problems. In particular, Garc{\'{\i}}a Azorero and Peral Alonso have shown in \cite{MR1083144,MR912211} that the corresponding $p$-Laplacian problem
\begin{equation} \label{1.1}
\left\{\begin{aligned}
- \Delta_p u & = \lambda\, |u|^{r - 2}\, u + |u|^{p^\ast - 2}\, u && \text{in } \Omega\\[10pt]
u & = 0 && \text{on } \bdry{\Omega},
\end{aligned}\right.
\end{equation}
where $1 < p < N$, $p < r < p^\ast$, and $p^\ast = Np/(N - p)$, has a nontrivial nonnegative solution for all sufficiently large $\lambda > 0$ when $N < p^2$ and for all $\lambda > 0$ when $N \ge p^2$. More recently, Barrios et al.\! \cite{MR3390088} have extended the Br{\'e}zis-Nirenberg result to critical nonlocal problems involving the fractional Laplacian operator and Biagi et al.\! \cite{BiDiVaVe} have extended it to critical mixed local-nonlocal problems.

On the other hand, it follows from a result of Costa and Wang \cite[Theorem 2]{MR2113928} that the number of solutions of problem \eqref{1.5} goes to infinity as $\lambda \to \infty$. The proof of this result is based on a priori estimates for solutions and makes use of certain orthogonal decompositions of $H^1_0(\Omega)$. These decompositions use eigenspaces of the linear operator $- \Delta$ on $H^1_0(\Omega)$. Consequently, this approach to multiplicity does not extend to problem \eqref{1.1}, where the nonlinear operator $- \Delta_p$ has no linear eigenspaces.

In the present paper we prove several abstract multiplicity theorems that are not based on splittings of the underlying space into linear subspaces and therefore can be used to obtain multiple nontrivial solutions of critical growth problems such as \eqref{1.1}. Our abstract results will allow us to show that, given any $m \in \N$, problem \eqref{1.1} has $m$ distinct pairs of nontrivial solutions for all sufficiently large $\lambda > 0$, with an explicit lower bound on $\lambda$. This lower bound will be in terms of an unbounded sequence $\seq{\lambda_m}$ of Dirichlet eigenvalues of $- \Delta_p$ in $\Omega$ introduced in Perera \cite{MR1998432} that is based on the $\Z_2$-cohomological index of Fadell and Rabinowitz \cite{MR0478189} (see \eqref{6}). More precisely, we have the following theorem.

\begin{theorem} \label{Theorem 1}
If $1 < p < N$ and $p < r < p^\ast$, then problem \eqref{1.1} has $m$ distinct pairs of nontrivial solutions $\pm u^\lambda_1,\dots,\pm u^\lambda_m$ for all
\begin{equation} \label{1.2}
\lambda > r \vol{\Omega}^{r/p - 1}\, \sup_{\tau > 0} \left[\frac{\lambda_m}{p \tau^{r - p}} - \frac{S^{N/p}}{N \tau^r} - \frac{\tau^{p^\ast - r}}{p^\ast \vol{\Omega}^{p/(N - p)}}\right],
\end{equation}
where $\vol{\Omega}$ denotes the volume of $\Omega$ and
\[
S = \inf_{u \in \D^{1,\,p}(\R^N) \setminus \set{0}}\, \frac{\int_{\R^N} |\nabla u|^p\, dx}{\left(\int_{\R^N} |u|^{p^\ast} dx\right)^{p/p^\ast}}
\]
is the best Sobolev constant. In particular, the number of solutions of problem \eqref{1.1} goes to infinity as $\lambda \to \infty$.
\end{theorem}

\begin{remark} \label{Remark 1.2}
We note that
\[
\frac{\lambda_m}{p \tau^{r - p}} - \frac{S^{N/p}}{N \tau^r} - \frac{\tau^{p^\ast - r}}{p^\ast \vol{\Omega}^{p/(N - p)}} \to - \infty \text{ as } \tau \to 0, \infty
\]
and hence the supremum in \eqref{1.2} is finite. This supremum goes to infinity as $m \to \infty$ since $\lambda_m \to \infty$.
\end{remark}

\begin{remark}
The case $r = p$ of problem \eqref{1.1} has been widely studied in the literature (see, e.g., \cite{MR1741848,MR2514055,MR956567,MR1009077} and the references therein).
\end{remark}

As another application of our abstract multiplicity results, consider the critical $(p,q)$-Laplacian problem
\begin{equation} \label{1}
\left\{\begin{aligned}
- \Delta_p u - \Delta_q u & = \lambda\, |u|^{r - 2}\, u + |u|^{p^\ast - 2}\, u && \text{in } \Omega\\[10pt]
u & = 0 && \text{on } \bdry{\Omega},
\end{aligned}\right.
\end{equation}
where $1 < q < p < N$, $p \le r < p^\ast$, and $\lambda > 0$ is a parameter. We have the following theorem.

\begin{theorem} \label{Theorem 1.4}
Problem \eqref{1} has $m$ distinct pairs of nontrivial solutions $\pm u^\lambda_1,\dots,\pm u^\lambda_m$ for all
\begin{equation} \label{2}
\lambda > r \vol{\Omega}^{r/p - 1}\, \sup_{\tau > 0} \left[\frac{\lambda_m}{p \tau^{r - p}} + \frac{\vol{\Omega}^{1 - q/p}}{q \tau^{r - q}}\, \lambda_m^{q/p} - \frac{S^{N/p}}{N \tau^r} - \frac{\tau^{p^\ast - r}}{p^\ast \vol{\Omega}^{p/(N - p)}}\right]
\end{equation}
in each of the following cases:
\begin{enumroman}
\item \label{Theorem 1.4.i} $p < r < p^\ast$,
\item \label{Theorem 1.4.ii} $r = p \le q^\ast$.
\end{enumroman}
In particular, the number of solutions of problem \eqref{1} goes to infinity as $\lambda \to \infty$ in these cases.
\end{theorem}

\begin{remark}
As in Remark \ref{Remark 1.2}, the supremum in \eqref{2} is finite and goes to infinity as $m \to \infty$.
\end{remark}

\begin{remark}
The existence of a nontrivial solution of problem \eqref{1} for all sufficiently large $\lambda > 0$ when $p < r < p^\ast$ was proved in Yin and Yang \cite{MR2890966}. Sufficient conditions for the existence of a nontrivial solution when $r = p$ and $\lambda > 0$ is either small or large were given in Candito et al.\! \cite{MR3415031}.
\end{remark}

The reason that eigenvalues based on the cohomological index appear in the estimates in Theorems \ref{Theorem 1} and \ref{Theorem 1.4} is that our abstract multiplicity results involve the cohomological index. We recall that this index is defined as follows. Let $W$ be a Banach space and let $\A$ denote the class of symmetric subsets of $W \setminus \set{0}$. For $A \in \A$, let $\overline{A} = A/\Z_2$ be the quotient space of $A$ with each $u$ and $-u$ identified, let $f : \overline{A} \to \RP^\infty$ be the classifying map of $\overline{A}$, and let $f^\ast : H^\ast(\RP^\infty) \to H^\ast(\overline{A})$ be the induced homomorphism of the Alexander-Spanier cohomology rings. The cohomological index of $A$ is defined by
\[
i(A) = \begin{cases}
\sup \set{m \ge 1 : f^\ast(\omega^{m-1}) \ne 0} & \text{if } A \ne \emptyset\\[5pt]
0 & \text{if } A = \emptyset,
\end{cases}
\]
where $\omega \in H^1(\RP^\infty)$ is the generator of the polynomial ring $H^\ast(\RP^\infty) = \Z_2[\omega]$. For example, the classifying map of the unit sphere $S^{m-1}$ in $\R^m,\, m \ge 1$ is the inclusion $\RP^{m-1} \incl \RP^\infty$, which induces isomorphisms on $H^q$ for $q \le m - 1$, so $i(S^{m-1}) = m$.

The cohomological index has the following so called piercing property, which is not shared by the Krasnoselskii’s genus. If $A, A_0, A_1 \in \A$ are closed and $\varphi : A \times [0,1] \to A_0 \cup A_1$ is a continuous map such that $\varphi(-u,t) = - \varphi(u,t)$ for all $(u,t) \in A \times [0,1]$, $\varphi(A \times [0,1])$ is closed, $\varphi(A \times \set{0}) \subset A_0$, and $\varphi(A \times \set{1}) \subset A_1$, then
\[
i(\varphi(A \times [0,1]) \cap A_0 \cap A_1) \ge i(A)
\]
(see \cite[Proposition (3.9)]{MR0478189}). The proofs of our abstract multiplicity results will make essential use of this special property of the cohomological index.

\section{Abstract multiplicity theorems}

In this section we state and prove our abstract multiplicity theorems. Let $W$ be a Banach space and let $E \in C^1(W,\R)$ be an even functional, i.e., $E(-u) = E(u)$ for all $u \in E$. Assume that there exists $c^\ast > 0$ such that for all $c \in (0,c^\ast)$, $E$ satisfies the \PS{c} condition, i.e., every sequence $\seq{u_j}$ in $W$ such that $E(u_j) \to c$ and $E'(u_j) \to 0$ has a strongly convergent subsequence. Let $\A^\ast$ denote the class of symmetric subsets of $W$ and let $\Gamma$ denote the group of odd homeomorphisms of $W$ that are the identity outside the set $\set{u \in W : 0 < E(u) < c^\ast}$. For $r > 0$, the pseudo-index of $M \in \A^\ast$ related to $i$, $S_r = \set{u \in W : \norm{u} = r}$, and $\Gamma$ is defined by
\[
i^\ast(M) = \min_{\gamma \in \Gamma}\, i(\gamma(M) \cap S_r)
\]
(see Benci \cite{MR84c:58014}). First we prove the following theorem.

\begin{theorem} \label{Theorem 1.1}
Let $C$ be a compact symmetric subset of $S = \set{u \in W : \norm{u} = 1}$ with $i(C) = m \ge 1$. Assume that the origin is a strict local minimizer of $E$ and that there exists $R > 0$ such that
\begin{equation} \label{2.2}
\sup_{u \in A}\, E(u) \le 0, \qquad \sup_{u \in X}\, E(u) < c^\ast,
\end{equation}
where $A = \set{Ru : u \in C}$ and $X = \set{tu : u \in A,\, 0 \le t \le 1}$. Let $0 < r < R$ be so small that
\begin{equation} \label{3}
\inf_{u \in S_r}\, E(u) > 0,
\end{equation}
let $\A_j^\ast = \set{M \in \A^\ast : M \text{ is compact and } i^\ast(M) \ge j}$, and set
\[
c_j^\ast := \inf_{M \in \A_j^\ast}\, \sup_{u \in M}\, E(u), \quad j = 1,\dots,m.
\]
Then $0 < c_1^\ast \le \dotsb \le c_m^\ast < c^\ast$, each $c_j^\ast$ is a critical value of $E$, and $E$ has $m$ distinct pairs of associated critical points.
\end{theorem}

\begin{proof}
For each $M \in \A_1^\ast$, since the identity map is in $\Gamma$,
\[
i(M \cap S_r) \ge i^\ast(M) \ge 1
\]
and hence $M \cap S_r \ne \emptyset$. It follows that
\[
c_1^\ast \ge \inf_{u \in S_r}\, E(u) > 0
\]
by \eqref{3}.

For $\gamma \in \Gamma$, consider the continuous map
\[
\varphi : A \times [0,1] \to W, \quad (u,t) \mapsto \gamma(tu).
\]
We have $\varphi(A \times [0,1]) = \gamma(X)$, which is compact. Since $\gamma$ is odd,
\[
\varphi(-u,t) = \gamma(-tu) = - \gamma(tu) = - \varphi(u,t) \quad \forall (u,t) \in A \times [0,1]
\]
and
\[
\varphi(A \times \set{0}) = \set{\gamma(0)} = \set{0}.
\]
Since $E \le 0$ on $A$ by \eqref{2.2}, $\gamma$ is the identity on $A$ and hence $\varphi(A \times \set{1}) = A$. Thus, applying the piercing property with
\[
A_0 = \set{u \in W : \norm{u} \le r}, \quad A_1 = \set{u \in W : \norm{u} \ge r}
\]
gives
\[
i(\gamma(X) \cap S_r) = i(\varphi(A \times [0,1]) \cap A_0 \cap A_1) \ge i(A) = i(C) \ge m,
\]
where the second equality holds since the mapping $C \to A,\, u \mapsto Ru$ is an odd homeomorphism. It follows that $i^\ast(X) \ge m$, so $X \in \A_m^\ast$ and hence
\[
c_m^\ast \le \sup_{u \in X}\, E(u) < c^\ast
\]
by \eqref{2.2}.

The rest now follows from standard arguments in critical point theory (see, e.g., Perera et al.\! \cite[Proposition 3.42]{MR2640827}).
\end{proof}

Next we briefly recall the construction of eigenvalues based on the cohomological index for $p$-Laplacian type operators in an abstract setting introduced in Perera \cite{MR4293883} (eigenvalues based on the cohomological index for the $p$-Laplacian were first introduced in Perera \cite{MR1998432}). Let $W$ be a uniformly convex Banach space with the dual $W^\ast$ and duality pairing $\dualp{\cdot}{\cdot}$. Recall that $f \in C(W,W^\ast)$ is a potential operator if there is a functional $F \in C^1(W,\R)$, called a potential for $f$, such that $F' = f$. We consider the nonlinear eigenvalue problem
\begin{equation} \label{1.3}
\Ap[u] = \lambda \Bp[u] \quad \text{in } W^\ast,
\end{equation}
where $\Ap, \Bp \in C(W,W^\ast)$ are potential operators satisfying the following hypotheses:
\begin{enumerate}
\item[$(A_1)$] $\Ap$ is $(p - 1)$-homogeneous and odd for some $p \in (1,\infty)$, i.e., $\Ap[(tu)] = |t|^{p - 2}\, t\, \Ap[u]$ for all $u \in W$ and $t \in \R$,
\item[$(A_2)$] $\dualp{\Ap[u]}{v} \le \norm{u}^{p - 1} \norm{v}$ for all $u, v \in W$, and equality holds if and only if $\alpha u = \beta v$ for some $\alpha, \beta \ge 0$, not both zero (in particular, $\dualp{\Ap[u]}{u} = \norm{u}^p$ for all $u \in W$),
\item[$(B_1)$] $\Bp$ is $(p - 1)$-homogeneous and odd, i.e., $\Bp[(tu)] = |t|^{p - 2}\, t\, \Bp[u]$ for all $u \in W$ and $t \in \R$,
\item[$(B_2)$] $\dualp{\Bp[u]}{u} > 0$ for all $u \in W \setminus \set{0}$, and $\dualp{\Bp[u]}{v} \le \dualp{\Bp[u]}{u}^{(p-1)/p} \dualp{\Bp[v]}{v}^{1/p}$ for all $u, v \in W$,
\item[$(B_3)$] $\Bp$ is a compact operator.
\end{enumerate}

Eigenvalues of problem \eqref{1.3} coincide with critical values of the $C^1$-functional
\begin{equation} \label{2.3}
\Psi(u) = \frac{1}{p\, J_p(u)}, \quad u \in S = \set{u \in W : p\, I_p(u) = 1},
\end{equation}
where
\[
I_p(u) = \frac{1}{p} \dualp{\Ap[u]}{u} = \frac{1}{p} \norm{u}^p, \quad J_p(u) = \frac{1}{p} \dualp{\Bp[u]}{u}
\]
are the potentials of $\Ap$ and $\Bp$ satisfying $I_p(0) = J_p(0) = 0$, respectively (see Perera \cite[Proposition 3.1]{MR4293883}). The following theorem was proved in Perera et al.\! \cite{MR2640827}.

\begin{theorem}[{\cite[Theorem 4.6]{MR2640827}}]
Let $\F$ be the class of symmetric subsets of $S$, let $\F_m = \set{M \in \F : i(M) \ge m}$, and set
\begin{equation} \label{6}
\lambda_m := \inf_{M \in \F_m}\, \sup_{u \in M}\, \Psi(u), \quad m \ge 1.
\end{equation}
Then
\begin{equation} \label{2.4}
\lambda_1 = \inf_{u \in S}\, \Psi(u) > 0
\end{equation}
is the first eigenvalue and $\lambda_1 \le \lambda_2 \le \cdots$ is an unbounded sequence of eigenvalues of the problem \eqref{1.3}.
\end{theorem}

Moreover, the following theorem was proved in Perera \cite{MR4293883}.

\begin{theorem}[{\cite[Theorem 1.3]{MR4293883}}] \label{Theorem 2.3}
If $\lambda_m < \lambda_{m+1}$, then the sublevel set
\[
\Psi^{\lambda_m} = \set{u \in S : \Psi(u) \le \lambda_m}
\]
has a compact symmetric subset $C$ of index $m$.
\end{theorem}

Now we consider the equation
\begin{equation} \label{16}
\Ap[u] + f(u) = \lambda\, g(u) + h(u) \quad \text{in } W^\ast,
\end{equation}
where $f, g, h \in C(W,W^\ast)$ are odd potential operators and $\lambda > 0$ is a parameter. We assume that the potentials $F$, $G$, and $H$ of $f$, $g$, and $h$ with $F(0) = G(0) = H(0) = 0$, respectively, satisfy the following hypotheses:
\begin{enumerate}
\item[$(H_1)$] for some constants $\alpha > 0$ and $1 < q < p$,
    \[
    0 \le F(u) \le \frac{\alpha}{q}\, \big(p\, I_p(u)\big)^{q/p} \quad \forall u \in W
    \]
\item[$(H_2)$] for some constants $\beta > 0$ and $r > p$,
    \[
    G(u) \ge \frac{\beta}{r}\, \big(p\, J_p(u)\big)^{r/p} \quad \forall u \in W
    \]
\item[$(H_3)$] $G(u) = \o(\norm{u}^p)$ as $u \to 0$,
\item[$(H_4)$] for some constants $\gamma > 0$ and $p^\ast > r$,
    \[
    H(u) \ge \frac{\gamma}{p^\ast}\, \big(p\, J_p(u)\big)^{p^\ast/p} \quad \forall u \in W
    \]
\item[$(H_5)$] $H(u) = \o(\norm{u}^p)$ as $u \to 0$,
\item[$(H_6)$] there exists $c^\ast > 0$ such that the variational functional
    \[
    E(u) = I_p(u) + F(u) - \lambda\, G(u) - H(u), \quad u \in W
    \]
    associated with equation \eqref{16} satisfies the \PS{c} condition for all $c \in (0,c^\ast)$.
\end{enumerate}
We have the following theorem.

\begin{theorem} \label{Theorem 6}
Assume $(H_1)$--$(H_6)$. Then equation \eqref{16} has $m$ distinct pairs of nontrivial solutions $\pm u^\lambda_1,\dots,\pm u^\lambda_m$ for all
\begin{equation} \label{17}
\lambda > \frac{r}{\beta}\, \sup_{\tau > 0} \left[\frac{\lambda_m}{p \tau^{r - p}} + \frac{\alpha}{q \tau^{r - q}}\, \lambda_m^{q/p} - \frac{c^\ast}{\tau^r} - \frac{\gamma \tau^{p^\ast - r}}{p^\ast}\right].
\end{equation}
\end{theorem}

\begin{proof}
We apply Theorem \ref{Theorem 1.1}. By $(H_1)$, $(H_3)$, and $(H_5)$,
\[
E(u) \ge \frac{1}{p} \norm{u}^p + \o(\norm{u}^p) \text{ as } u \to 0,
\]
so the origin is a strict local minimizer of $E$.

By increasing $m$ if necessary, we may assume that $\lambda_m < \lambda_{m+1}$. Then $\Psi^{\lambda_m}$ has a compact symmetric subset $C$ of index $m$ by Theorem \ref{Theorem 2.3}. Let $R > 0$ and let $A$ and $X$ be as in Theorem \ref{Theorem 1.1}. For $u \in C \subset \Psi^{\lambda_m}$, $(H_1)$, $(H_2)$, $(H_4)$, and \eqref{2.3} give
\begin{gather*}
F(Ru) \le \frac{\alpha R^q}{q}\, \big(p\, I_p(u)\big)^{q/p} = \frac{\alpha R^q}{q},\\[5pt]
G(Ru) \ge \frac{\beta R^r}{r}\, \big(p\, J_p(u)\big)^{r/p} = \frac{\beta R^r}{r\, \Psi(u)^{r/p}} \ge \frac{\beta R^r}{r\, \lambda_m^{r/p}},\\[5pt]
H(Ru) \ge \frac{\gamma R^{p^\ast}}{p^\ast}\, \big(p\, J_p(u)\big)^{p^\ast/p} = \frac{\gamma R^{p^\ast}}{p^\ast\, \Psi(u)^{p^\ast/p}} \ge \frac{\gamma R^{p^\ast}}{p^\ast \lambda_m^{p^\ast/p}}.
\end{gather*}
So
\begin{equation} \label{19}
E(Ru) = \frac{R^p}{p} + F(Ru) - \lambda\, G(Ru) - H(Ru) \le \frac{R^p}{p} + \frac{\alpha R^q}{q} - \frac{\lambda \beta R^r}{r\, \lambda_m^{r/p}} - \frac{\gamma R^{p^\ast}}{p^\ast \lambda_m^{p^\ast/p}}.
\end{equation}
Since $\gamma > 0$ and $p^\ast > r > p > q$, it follows from this that the first inequality in \eqref{2.2} holds if $R$ is sufficiently large.

For $u \in A$ and $0 \le t \le 1$, \eqref{19} gives
\begin{equation} \label{20}
E(tu) \le \frac{t^p R^p}{p} + \frac{\alpha t^q R^q}{q} - \frac{\lambda \beta t^r R^r}{r\, \lambda_m^{r/p}} - \frac{\gamma t^{p^\ast} R^{p^\ast}}{p^\ast \lambda_m^{p^\ast/p}} = \frac{\tau^p}{p}\, \lambda_m + \frac{\alpha \tau^q}{q}\, \lambda_m^{q/p} - \frac{\lambda \beta \tau^r}{r} - \frac{\gamma \tau^{p^\ast}}{p^\ast},
\end{equation}
where $\tau = tR/\lambda_m^{1/p}$. The supremum of the last expression over all $\tau \ge 0$ is attained at some $\tau_0 > 0$. Then \eqref{20} gives
\begin{equation} \label{21}
\sup_{u \in X}\, E(u) \le \frac{\tau_0^p}{p}\, \lambda_m + \frac{\alpha \tau_0^q}{q}\, \lambda_m^{q/p} - \frac{\lambda \beta \tau_0^r}{r} - \frac{\gamma \tau_0^{p^\ast}}{p^\ast}.
\end{equation}
We have
\[
\nu_m := \frac{r}{\beta}\, \sup_{\tau > 0} \left[\frac{\lambda_m}{p \tau^{r - p}} + \frac{\alpha}{q \tau^{r - q}}\, \lambda_m^{q/p} - \frac{c^\ast}{\tau^r} - \frac{\gamma \tau^{p^\ast - r}}{p^\ast}\right] \ge \frac{r}{\beta} \left[\frac{\lambda_m}{p \tau_0^{r - p}} + \frac{\alpha}{q \tau_0^{r - q}}\, \lambda_m^{q/p} - \frac{c^\ast}{\tau_0^r} - \frac{\gamma \tau_0^{p^\ast - r}}{p^\ast}\right]
\]
and hence
\begin{equation} \label{22}
\frac{\tau_0^p}{p}\, \lambda_m + \frac{\alpha \tau_0^q}{q}\, \lambda_m^{q/p} - \frac{\nu_m \beta \tau_0^r}{r} - \frac{\gamma \tau_0^{p^\ast}}{p^\ast} \le c^\ast.
\end{equation}
Since $\lambda > \nu_m$ by \eqref{17}, it follows from \eqref{21} and \eqref{22} that the second inequality in \eqref{2.2} also holds. Theorem \ref{Theorem 1.1} now gives $m$ distinct pairs of nontrivial critical points of $E$.
\end{proof}

Finally we consider the equation
\begin{equation} \label{2.10}
\Ap[u] + f(u) = \lambda \Bp[u] + h(u) \quad \text{in } W^\ast,
\end{equation}
where $f, h \in C(W,W^\ast)$ are odd potential operators and $\lambda > 0$ is a parameter. We assume that the potentials $F$ and $H$ of $f$ and $h$ with $F(0) = H(0) = 0$, respectively, satisfy the following hypotheses:
\begin{enumerate}
\item[$(H_1)$] for some constants $\alpha_0, \alpha > 0$ and $1 < q < p$,
    \[
    \frac{\alpha_0}{q}\, \big(p\, J_p(u)\big)^{q/p} \le F(u) \le \frac{\alpha}{q}\, \big(p\, I_p(u)\big)^{q/p} \quad \forall u \in W,
    \]
\item[$(H_2)$] for some constants $\gamma > 0$ and $p^\ast > p$,
    \[
    H(u) \ge \frac{\gamma}{p^\ast}\, \big(p\, J_p(u)\big)^{p^\ast/p} \quad \forall u \in W,
    \]
\item[$(H_3)$] $H(u) = \o(\norm{u}^p)$ as $u \to 0$,
\item[$(H_4)$] there exists $c^\ast > 0$ such that the variational functional
    \[
    E(u) = I_p(u) + F(u) - \lambda J_p(u) - H(u), \quad u \in W
    \]
    associated with equation \eqref{2.10} satisfies the \PS{c} condition for all $c \in (0,c^\ast)$.
\end{enumerate}
We have the following theorem.

\begin{theorem} \label{Theorem 5}
Assume $(H_1)$--$(H_4)$. Then equation \eqref{2.10} has $m$ distinct pairs of nontrivial solutions $\pm u^\lambda_1,\dots,\pm u^\lambda_m$ for all
\begin{equation} \label{2.8}
\lambda > \lambda_m + p\, \sup_{\tau > 0} \left[\frac{\alpha}{q \tau^{p - q}}\, \lambda_m^{q/p} - \frac{c^\ast}{\tau^p} - \frac{\gamma \tau^{p^\ast - p}}{p^\ast}\right].
\end{equation}
\end{theorem}

\begin{proof}
We apply Theorem \ref{Theorem 1.1}. First we show that the origin is a strict local minimizer of $E$. By $(H_1)$ and the elementary inequality $a^2 - ab \ge - b^2/4$ for all $a, b \in \R$,
\[
F(u) - \lambda J_p(u) \ge \frac{\alpha_0}{q}\, \big(p\, J_p(u)\big)^{q/p} - \lambda J_p(u) \ge - \frac{q\, \lambda^2}{4p^2\, \alpha_0}\, \big(p\, J_p(u)\big)^{2 - q/p}.
\]
Since $J_p(u) \le I_p(u)/\lambda_1 = \norm{u}^p/p\, \lambda_1$ by \eqref{2.4}, this together with $(H_3)$ gives
\[
E(u) \ge \frac{1}{p} \left(1 - \frac{q\, \lambda^2}{4p\, \alpha_0\, \lambda_1^{2 - q/p}}\, \norm{u}^{p - q} + \o(1)\right) \norm{u}^p \text{ as } u \to 0,
\]
from which the desired conclusion follows since $q < p$.

By increasing $m$ if necessary, we may assume that $\lambda_m < \lambda_{m+1}$. Then $\Psi^{\lambda_m}$ has a compact symmetric subset $C$ of index $m$ by Theorem \ref{Theorem 2.3}. Let $R > 0$ and let $A$ and $X$ be as in Theorem \ref{Theorem 1.1}. For $u \in S$,
\begin{equation} \label{2.9}
E(Ru) = R^p\, \big(I_p(u) - \lambda J_p(u)\big) + F(Ru) - H(Ru) = \frac{R^p}{p} \left(1 - \frac{\lambda}{\Psi(u)}\right) + F(Ru) - H(Ru)
\end{equation}
by \eqref{2.3}. For $u \in C \subset \Psi^{\lambda_m}$, $(H_1)$, $(H_2)$, and \eqref{2.3} give
\begin{gather*}
F(Ru) \le \frac{\alpha R^q}{q}\, \big(p\, I_p(u)\big)^{q/p} = \frac{\alpha R^q}{q},\\[5pt]
H(Ru) \ge \frac{\gamma R^{p^\ast}}{p^\ast}\, \big(p\, J_p(u)\big)^{p^\ast/p} = \frac{\gamma R^{p^\ast}}{p^\ast\, \Psi(u)^{p^\ast/p}} \ge \frac{\gamma R^{p^\ast}}{p^\ast \lambda_m^{p^\ast/p}}.
\end{gather*}
These inequalities together with \eqref{2.9} give
\begin{equation} \label{2.13}
E(Ru) \le \frac{R^p}{p} \left(1 - \frac{\lambda}{\lambda_m}\right) + \frac{\alpha R^q}{q} - \frac{\gamma R^{p^\ast}}{p^\ast \lambda_m^{p^\ast/p}}.
\end{equation}
Since $\gamma > 0$ and $p^\ast > p > q$, it follows from this that the first inequality in \eqref{2.2} holds if $R$ is sufficiently large.

For $u \in A$ and $0 \le t \le 1$, \eqref{2.13} gives
\begin{equation} \label{2.11}
E(tu) \le \frac{t^p R^p}{p} \left(1 - \frac{\lambda}{\lambda_m}\right) + \frac{\alpha t^q R^q}{q} - \frac{\gamma t^{p^\ast} R^{p^\ast}}{p^\ast \lambda_m^{p^\ast/p}} = \frac{\tau^p}{p}\, (\lambda_m - \lambda) + \frac{\alpha \tau^q}{q}\, \lambda_m^{q/p} - \frac{\gamma \tau^{p^\ast}}{p^\ast},
\end{equation}
where $\tau = tR/\lambda_m^{1/p}$. The supremum of the last expression over all $\tau \ge 0$ is attained at some $\tau_0 > 0$. Then \eqref{2.11} gives
\begin{equation} \label{2.14}
\sup_{u \in X}\, E(u) \le \frac{\tau_0^p}{p}\, (\lambda_m - \lambda) + \frac{\alpha \tau_0^q}{q}\, \lambda_m^{q/p} - \frac{\gamma \tau_0^{p^\ast}}{p^\ast}.
\end{equation}
We have
\[
\nu_m := p\, \sup_{\tau > 0} \left[\frac{\alpha}{q \tau^{p - q}}\, \lambda_m^{q/p} - \frac{c^\ast}{\tau^p} - \frac{\gamma \tau^{p^\ast - p}}{p^\ast}\right] \ge p \left[\frac{\alpha}{q \tau_0^{p - q}}\, \lambda_m^{q/p} - \frac{c^\ast}{\tau_0^p} - \frac{\gamma \tau_0^{p^\ast - p}}{p^\ast}\right]
\]
and hence
\begin{equation} \label{2.12}
- \frac{\tau_0^p}{p}\, \nu_m + \frac{\alpha \tau_0^q}{q}\, \lambda_m^{q/p} - \frac{\gamma \tau_0^{p^\ast}}{p^\ast} \le c^\ast.
\end{equation}
Since $\lambda_m - \lambda < - \nu_m$ by \eqref{2.8}, it follows from \eqref{2.14} and \eqref{2.12} that the second inequality in \eqref{2.2} also holds. Theorem \ref{Theorem 1.1} now gives $m$ distinct pairs of nontrivial critical points of $E$.
\end{proof}

\section{Applications}

In this section we prove Theorems \ref{Theorem 1} and \ref{Theorem 1.4}. We apply Theorems \ref{Theorem 6} and \ref{Theorem 5} with $W = W^{1,\,p}_0(\Omega)$ and the operators $\Ap, \Bp, f, g, h \in C(W^{1,\,p}_0(\Omega),W^{-1,\,p'}(\Omega))$ given by
\[
\dualp{\Ap[u]}{v} = \int_\Omega |\nabla u|^{p - 2}\, \nabla u \cdot \nabla v\, dx, \quad \dualp{\Bp[u]}{v} = \int_\Omega |u|^{p - 2}\, uv\, dx
\]
and
\[
\dualp{f(u)}{v} = \int_\Omega |\nabla u|^{q-2}\, \nabla u \cdot \nabla v\, dx, \quad \dualp{g(u)}{v} = \int_\Omega |u|^{r - 2}\, uv\, dx, \quad \dualp{h(u)}{v} = \int_\Omega |u|^{p^\ast - 2}\, uv\, dx
\]
for $u, v \in W^{1,\,p}_0(\Omega)$. The corresponding potentials are given by
\[
I_p(u) = \frac{1}{p} \int_\Omega |\nabla u|^p\, dx, \quad J_p(u) = \frac{1}{p} \int_\Omega |u|^p\, dx
\]
and
\[
F(u) = \frac{1}{q} \int_\Omega |\nabla u|^q\, dx, \quad G(u) = \frac{1}{r} \int_\Omega |u|^r\, dx, \quad H(u) = \frac{1}{p^\ast} \int_\Omega |u|^{p^\ast}\, dx,
\]
respectively. It follows from the H\"{o}lder inequality that
\begin{gather}
\label{3.1} F(u) \le \frac{\vol{\Omega}^{1 - q/p}}{q} \left(\int_\Omega |\nabla u|^p\, dx\right)^{q/p} = \frac{\vol{\Omega}^{1 - q/p}}{q}\, \big(p\, I_p(u)\big)^{q/p},\\[5pt]
\label{3.2} G(u) \ge \frac{1}{r \vol{\Omega}^{r/p - 1}} \left(\int_\Omega |u|^p\, dx\right)^{r/p} = \frac{1}{r \vol{\Omega}^{r/p - 1}}\, \big(p\, J_p(u)\big)^{r/p},\\[5pt]
\label{3.3} H(u) \ge \frac{1}{p^\ast \vol{\Omega}^{p/(N - p)}} \left(\int_\Omega |u|^p\, dx\right)^{p^\ast/p} = \frac{1}{p^\ast \vol{\Omega}^{p/(N - p)}}\, \big(p\, J_p(u)\big)^{p^\ast/p}.
\end{gather}
The Sobolev inequality gives
\[
G(u) = \o(\norm{u}^p) \text{ as } u \to 0
\]
when $r > p$ and
\[
H(u) = \o(\norm{u}^p) \text{ as } u \to 0.
\]
As is well-known, the associated variational functionals
\[
E(u) = \frac{1}{p} \int_\Omega |\nabla u|^p\, dx - \frac{\lambda}{r} \int_\Omega |u|^r\, dx - \frac{1}{p^\ast} \int_\Omega |u|^{p^\ast}\, dx, \quad u \in W^{1,\,p}_0(\Omega)
\]
and
\[
E(u) = \frac{1}{p} \int_\Omega |\nabla u|^p\, dx + \frac{1}{q} \int_\Omega |\nabla u|^q\, dx - \frac{\lambda}{r} \int_\Omega |u|^r\, dx - \frac{1}{p^\ast} \int_\Omega |u|^{p^\ast}\, dx, \quad u \in W^{1,\,p}_0(\Omega)
\]
satisfy the \PS{c} condition for all
\[
c < \frac{1}{N}\, S^{N/p}
\]
when $p \le r < p^\ast$.

\begin{proof}[Proof of Theorem \ref{Theorem 1}]
We apply Theorem \ref{Theorem 6} with $f = 0$ and $\alpha = 0$. By \eqref{3.2} and \eqref{3.3}, $(H_2)$ and $(H_4)$ hold with
\[
\beta = \frac{1}{\vol{\Omega}^{r/p - 1}}, \quad \gamma = \frac{1}{\vol{\Omega}^{p/(N - p)}},
\]
respectively. The conclusion follows from Theorem \ref{Theorem 6}.
\end{proof}

\begin{proof}[Proof of Theorem \ref{Theorem 1.4}]
\ref{Theorem 1.4.i} Let $p < r < p^\ast$. We apply Theorem \ref{Theorem 6}. By \eqref{3.1}--\eqref{3.3}, $(H_1)$, $(H_2)$, and $(H_4)$ hold with
\[
\alpha = \vol{\Omega}^{1 - q/p}, \quad \beta = \frac{1}{\vol{\Omega}^{r/p - 1}}, \quad \gamma = \frac{1}{\vol{\Omega}^{p/(N - p)}},
\]
respectively. The conclusion follows from Theorem \ref{Theorem 6}.

\ref{Theorem 1.4.ii} Let $r = p \le q^\ast$. We apply Theorem \ref{Theorem 5}. Since $p \le q^\ast$, the first inequality in $(H_1)$ holds for some $\alpha_0 > 0$ by the Sobolev embedding theorem. By \eqref{3.1} and \eqref{3.3}, the second inequality in $(H_1)$ and $(H_2)$ hold with
\[
\alpha = \vol{\Omega}^{1 - q/p}, \quad \gamma = \frac{1}{\vol{\Omega}^{p/(N - p)}},
\]
respectively. The conclusion follows from Theorem \ref{Theorem 5}.
\end{proof}

\def\cprime{$''$}

\end{document}